\documentclass[12pt]{article}
\usepackage{amssymb}
\usepackage{amsmath}
\usepackage{amsthm}
\usepackage{amsfonts}

\usepackage{a4wide}
\usepackage{longtable}
\usepackage{multirow}

\usepackage[latin1]{inputenc}
\usepackage[english]{babel}
\usepackage[T1]{fontenc}
\usepackage{amssymb,amsmath,amsthm,amstext,amsfonts,latexsym}

\usepackage{pgf,tikz}
\usepackage{mathrsfs}
\usetikzlibrary{arrows}

\newtheorem{lemma}{Lemma}[section]

\newtheorem{theorem}{Theorem}[section]

\theoremstyle{definition}

\theoremstyle{remark}
\newtheorem{remark}{Remark}[section]

\usepackage[all]{xy}
\input xy
\xyoption {all}
\usepackage[pdftex,             %
    bookmarks=true,             
    hypertexnames=false,        
    bookmarksnumbered=true,     
    pdfpagemode=None,           
    pdfstartview=FitH,          
    pdfpagelayout=SinglePage,   
    colorlinks=true,            
    urlcolor=blue,              
    citecolor=blue,             
    pdfborder={0 0 0}           
    ]{hyperref}

\date{}

\begin{document}

\title{ Fields $\mathbb{Q}(\sqrt[3]{d},\zeta_3)$  whose  $3$-class group is of
type $(9,3)$ }

\author{S. AOUISSI,  M. C. ISMAILI, M. TALBI  and A. AZIZI}

\maketitle

\medskip
\noindent
\textbf{Abstract:}
Let $\mathrm{k}=\mathbb{Q}(\sqrt[3]{d},\zeta_3)$, with $d$ a cube-free positive integer.
Let $C_{\mathrm{k},3}$ be the \(3\)-class group of \(\mathrm{k}\).
With the aid of genus theory,
arithmetic properties of the pure cubic field $\mathbb{Q}(\sqrt[3]{d})$
and some results on the $3$-class group $C_{\mathrm{k},3}$,
we determine all integers $d$ such that $C_{\mathrm{k},3}\simeq\mathbb{Z}/9\mathbb{Z}\times\mathbb{Z}/3\mathbb{Z}$.

 
\section{Introduction}
\paragraph*{} 
Let $d$ be a cube-free positive integer,
$\mathrm{k}=\mathbb{Q}(\sqrt[3]{d},\zeta_3)$,
and $C_{\mathrm{k},3}$ be the \(3\)-class group of $\mathrm{k}$.
A number of researchers have studied the $3$-class group $C_{\mathrm{k},3}$
and the calculation of its rank.
Calegari and Emerton \cite[Lemma 5.11]{Cal-Emer} proved that the rank of the $3$-class group of $\mathbb{Q}(\sqrt[3]{p},\zeta_3)$,
with a prime $p\equiv 1\pmod 9$,
is equal to two if $9$ divides the class number of $\mathbb{Q}(\sqrt[3]{p})$.
The converse of the Calegari-Emerton result was proved by Frank Gerth III in \cite[Theorem 1, p. 471]{GERTH3}.
\paragraph*{} 
The purpose of this paper is to classify all integers $d$ for which $C_{\mathrm{k},3}$ is of type $(9,3)$,
i.e. $C_{\mathrm{k},3}\simeq\mathbb{Z}/9\mathbb{Z}\times\mathbb{Z}/3\mathbb{Z}$.
This investigation can be viewed as a continuation of the previous more general works
\cite[Lemma 5.11]{Cal-Emer} and \cite[Theorem 1, p. 471]{GERTH3}.
Effectively, we shall prove the following main theorem:

\begin{theorem}\label{T39}
Let $\Gamma=\mathbb{Q}(\sqrt[3]{d})$ be a pure cubic field,
where $d\ge 2$ is a cube-free integer,
and let $\mathrm{k}=\mathbb{Q}(\sqrt[3]{d},\zeta_3)$ be its normal closure.
Denote by $u$ the index of the subgroup generated by the units of intermediate fields of the extension $\mathrm{k}/\mathbb{Q}$
in the group of units of $\mathrm{k}$.
\begin{itemize}
\item[1)]
If the field $\mathrm{k}$ has a $3$-class group of type $(9,3)$,
then $d=p^{e}$, where $p$ is a prime number congruent to $1\pmod 9$ and $e=1$ or $2$.
\item[2)]
If $p$ is a prime number congruent to $1\pmod 9$,
$9$ divides the class number of $\Gamma$ exactly, and $u=1$,
then the $3$-class group of $\mathrm{k}$ is of type $(9,3).$
\end{itemize}
\end{theorem}
This result will be underpinned by numerical examples
obtained with the computational number theory system PARI \cite{PARI}
in \S\ \ref{Appendix}.
In section \ref{sec2}, where Theorem \ref{T39} is proved,
we only state results that will be needed in this paper.
More information on $3$-class groups can be found in \cite{GERTH1} and \cite{GERTH2}.
For the prime ideal factorization in the pure cubic field $\mathbb{Q}(\sqrt[3]{d})$,
we refer the reader to the papers \cite{DED}, \cite{B-C}, \cite{BC2} and \cite{Markf}.
For the prime factorization rules of the third cyclotomic field $\mathbb{Q}(\zeta_3)$,
we refer the reader to \cite[Chap. 9, Sec. 1, Propositions 9.1.1--4, pp. 109-111]{Clas}. 
\begin{flushleft}
\textbf{Notations:}
\end{flushleft}
 $\bullet$ The lower case letter $p$, respectively $q$, will denote a prime number congruent to $1$, respectively $-1$, modulo $3$;\\
 $\bullet$ $\Gamma= \mathbb{Q}(\sqrt[3]{d})$ : a pure cubic field, where $d\ge 2$ is a cube-free integer;\\
 $\bullet$ $\mathrm{k}_0= \mathbb{Q}(\zeta_3)$ : the cyclotomic field, where $\zeta_3=e^{2i\pi/3}$;\\
 $\bullet$ $\mathrm{k}=\Gamma(\zeta_3)$ : the normal closure of $\Gamma$;\\
 $\bullet$ $\Gamma^{\prime}$ and $\Gamma^{\prime\prime}$ : the two conjugate cubic fields of $\Gamma$, contained in $\mathrm{k}$;\\
 $\bullet$ $u$ : the index of the subgroup $E_0$ generated by the units of intermediate fields of the extension $\mathrm{k}/\mathbb{Q}$ in the group of units of $\mathrm{k}$;\\
$\bullet$ $\langle\tau\rangle=\operatorname{Gal}\left(\mathrm{k}/\Gamma\right)$,
 such that $\tau^2=id$, $\tau(\zeta_3)=\zeta_3^2$, and $\tau(\sqrt[3]{d})=\sqrt[3]{d}$;\\
 $\bullet$ $\langle\sigma\rangle=\operatorname{Gal}\left(\mathrm{k}/\mathrm{k}_0\right)$,
 such that $\sigma^3=id$, $\sigma(\zeta_3)=\zeta_3$, and $\sigma(\sqrt[3]{d})=\zeta_3\sqrt[3]{d}$;\\
 $\bullet$ $\lambda=1-\zeta_3$ is a prime element above $3$ of $\mathrm{k}_0$;\\
 $\bullet$ $q^*=1$ or $0$, according to whether \(\zeta_{3}\) is norm of an element of $\mathrm{k}$ or not;\\
 $\bullet$ \(t\) : the number of prime ideals ramified in \(\mathrm{k}/\mathrm{k}_{0}\);\\
 $\bullet$ For an algebraic number field $L$:
  \begin{itemize}
  \item[$-$] $\mathcal{O}_{L}$, $E_{L}$  : the ring of integers of $L$, and the group of units of $L$;
  \item[$-$] $C_{L,3}$, $h_{L}$ :  the $3$-class group of $L$, and the class number of $L$;
  \item[$-$] $L_3^{(1)}$, $L^{*}$  : the Hilbert $3$-class field of $L$, and the absolute genus field of $L$.
  \end{itemize}


 \section{Fields $\mathbb{Q}(\sqrt[3]{d},\zeta_3)$  whose  $3$-Class Group is of
Type \((9,3)\)}\label{sec2}

\subsection{Preliminary results}

In \cite[Chap. 7, pp. 87--96]{ISCHIDA},
Ishida has explicitly given the genus field of any pure field.
For the pure cubic field $\Gamma=\mathbb{Q}(\sqrt[3]{d})$,
where $d$ is a cube-free natural number,
we have the following theorem.

\begin{theorem}\label{GenusGamma}
Let $\Gamma= \mathbb{Q}(\sqrt[3]{d})$ be a pure cubic field, where $d\ge 2$ is a cube-free integer,
and let $p_1,\ldots,p_r$ be all prime divisors of $d$ such that $p_{i}$ is congruent to $1\pmod 3$ for each $i\in\lbrace 1,2,\ldots,r\rbrace$.
Let $\Gamma^{\ast}$ be the absolute genus field of $\Gamma$, then
\begin{center}
$\Gamma^{\ast}=\prod_{i=1}^{r} M(p_i)\cdot\Gamma,$
\end{center}
where $M(p_i)$ denotes the unique subfield of degree $3$ of the cyclotomic field $\mathbb{Q}(\zeta_{p_i})$.
The genus number of $\Gamma$ is given by $g_{\Gamma}=3^{r}$.
\end{theorem}

\begin{remark} 
(1)
If no prime $p\equiv 1\pmod 3$ divides $d$, i.e. $r=0$, then $\Gamma^{\ast}=\Gamma$.\\
(2)
For any value $r\geq 0$, $\Gamma^{\ast}$ is contained in the Hilbert $3$-class field $\Gamma_3^{(1)}$ of $\Gamma $.\\
(3)
The cubic field $M(p)$ is determined explicitly in \cite[\S\ 4, Proposition 1, p. 11]{HONDA}.
\end{remark}

Assuming that $h_{\Gamma}$ is divisible exactly by $9$,
we can explicitly construct the absolute genus field $\Gamma^{*}$ as follows: 

\begin{lemma}\label{2p}
Let $\Gamma=\mathbb{Q}(\sqrt[3]{d})$ be a pure cubic field,
where $d\ge 2$ is a cube-free integer.
If $h_{\Gamma}$ is exactly divisible by $9$,
then there are at most two primes congruent to $1\pmod 3$ dividing $d$.
\end{lemma}
\begin{proof}
If $p_1,\ldots,p_r$ are all prime numbers congruent to $1\pmod 3$ dividing $d$, then $3^{r}|h_\Gamma$.
Therefore, if $h_{\Gamma}$ is exactly divisible by $9$, then $r\le 2$.
So there are two primes $p_1$ and $p_2$ dividing $d$ such that $p_{i}\equiv 1\pmod 3$ for $i\in\lbrace 1,2\rbrace$,
or there is only one prime $p\equiv 1\pmod 3$ with $p|d$,
or there is no prime $p\equiv 1\pmod 3$ such that $p|d$.
\end{proof}

\begin{lemma}\label{2prank}
Let $\Gamma=\mathbb{Q}(\sqrt[3]{d})$ be a pure cubic field,
where $d\ge 2$ is a cube-free integer.
If $h_{\Gamma}$ is exactly divisible by $9$
and if the integer $d$ is divisible by two primes $p_1$ and $p_2$ such that $p_i\equiv 1\pmod 3$ for $i\in\lbrace 1,2\rbrace$,
then $\Gamma^{\ast}=\Gamma_3^{(1)}$, $(\Gamma^{\prime})^{\ast}={\Gamma^{\prime}}_3^{(1)}$ and $(\Gamma^{\prime\prime})^{\ast}={\Gamma^{\prime\prime}}_3^{(1)}$.
Furthermore, $\mathrm{k}\cdot\Gamma_3^{(1)}=\mathrm{k}\cdot{\Gamma^{\prime}}_3^{(1)}=\mathrm{k}\cdot{\Gamma^{\prime\prime}}_3^{(1)}$. 
\end{lemma}
\begin{proof}
If $h_{\Gamma}$ is exactly divisible by $9$
and $d$ is divisible by two prime numbers $p_1$ and $p_2$ which are congruent to $1\pmod 3$,
then $g_\Gamma=9$ so $\Gamma^{\ast}=\Gamma_3^{(1)}=\Gamma\cdot M(p_1)\cdot M(p_2)$,
where $M(p_{1})$ (respectively $M(p_{2})$) is the unique cubic subfield of $\mathbb{Q}(\zeta_{p_{1}})$ (respectively $\mathbb{Q}(\zeta_{p_{2}})$).
The equations 
$$\begin{cases}
(\Gamma^{\prime})^{\ast}=\Gamma^{\prime}\cdot M(p_1)\cdot M(p_2)={\Gamma^{\prime}}_3^{(1)},\\
(\Gamma^{\prime\prime})^{\ast}=\Gamma^{\prime\prime}\cdot M(p_1)\cdot M(p_2)={\Gamma^{\prime\prime}}_3^{(1)}  
\end{cases}$$
can be deduced by the fact that in the general case we have
$$ \begin{cases}
(\Gamma^{\prime})^{\ast}=\Gamma^{\prime}\prod\limits_{i=1}^{r} M(p_i),\\
(\Gamma^{\prime\prime})^{\ast}=\Gamma^{\prime\prime} \prod\limits_{i=1}^{r} M(p_i),
\end{cases} $$
where $p_{i}$, for each $1\leq i\leq r$, is a prime divisor of $d$ such that $p_i\equiv 1\pmod 3$.
From the fact that $h_{\Gamma}$ is exactly divisible by $9$,
we conclude that $h_{\Gamma^{\prime}}$ is exactly divisible by $9$ and $h_{\Gamma^{\prime\prime}}$ is exactly divisible by $9$,
because $\Gamma$, $\Gamma^{\prime}$ and $\Gamma^{\prime\prime}$ are isomorphic.\\
Moreover, 
$$\begin{cases}
 \mathrm{k}\cdot\Gamma_3^{(1)}=\mathrm{k}\cdot\Gamma\cdot M(p_1)\cdot M(p_2) = \mathrm{k}\cdot M(p_1)\cdot M(p_2), \\
  \mathrm{k}\cdot{\Gamma^{\prime}}_3^{(1)}=\mathrm{k}\cdot\Gamma^{\prime}\cdot M(p_1)\cdot M(p_2) = \mathrm{k}\cdot M(p_1)\cdot M(p_2),\\
    \mathrm{k}\cdot{\Gamma^{\prime\prime}}_3^{(1)}=\mathrm{k}\cdot\Gamma^{\prime\prime}\cdot M(p_1)\cdot M(p_2) = \mathrm{k}\cdot M(p_1)\cdot M(p_2).
\end{cases}$$
Hence, $\mathrm{k}\cdot\Gamma_3^{(1)}=\mathrm{k}\cdot{\Gamma^{\prime}}_3^{(1)}=\mathrm{k}\cdot{\Gamma^{\prime\prime}}_3^{(1)}$.
\end{proof}

Now, let $u$ be the index of units defined in the above notations.
We assume that $h_{\Gamma}$ is exactly divisible by $9$ and $u=1$.
From \cite[\S\ 14, Theorem 14.1, p. 232]{B-C},
we have $h_{\mathrm{k}}=\dfrac{u}{3}\cdot h_{\Gamma}^{2}$, whence $h_{\mathrm{k}}$ is exactly divisible by $27$.
The structure of the $3$-class group $C_{\mathrm{k},3}$ is described by the following Lemma:

\begin{lemma}\label{39}
Let $\Gamma$ be a pure cubic field,
\(\mathrm{k}\) its normal closure,
and \(u\) be the index of units defined in the notations above, then
$$C_{\mathrm{k},3}\simeq \mathbb{Z}/9\mathbb{Z}\times \mathbb{Z}/3\mathbb{Z}
\quad \Leftrightarrow \quad \lbrack C_{\Gamma,3}\simeq \mathbb{Z}/9\mathbb{Z}\quad
\text{and} \quad u=1 \rbrack.$$
\end{lemma}
Lemma \ref{39} will be underpinned in section \ref{Appendix} by numerical examples
obtained with the computational number theory system PARI \cite{PARI}.
\begin{proof}
Assume that $C_{\mathrm{k},3}\simeq\mathbb{Z}/9\mathbb{Z}\times\mathbb{Z}/3\mathbb{Z}$.
Let $h_{\Gamma,3}$(respectively, $h_{\mathrm{k},3}$) be the $3$-class number of $\Gamma$ (respectively, $\mathrm{k}$),
then $h_{\mathrm{k},3}=27$.
According to \cite[\S\ 14, Theorem 14.1, p. 232]{B-C}, we get
$27=\frac{u}{3}\cdot h_{\Gamma,3}^{2}$ with $u\in\lbrace 1,3\rbrace$, and thus $u=1$,
because otherwise $27$ would be a square in $\mathbb{N}$, which is a contradiction.
Thus \(h_{\Gamma,3}^{2}=81\) and \(h_{\Gamma,3}=9\).\\
Let $C_{\mathrm{k},3}^{+} = \lbrace \mathcal{A} \in C_{\mathrm{k},3} \mid \mathcal{A}^{\tau} = \mathcal{A} \rbrace$
and $C_{\mathrm{k},3}^{-} = \lbrace \mathcal{A} \in C_{\mathrm{k},3} \mid \mathcal{A}^{\tau} = \mathcal{A}^{-1} \rbrace$.
According to \cite[\S\ 2, Lemmas 2.1 and 2.2, p. 53]{GERTH2},  we have
\( C_{\mathrm{k},3}\simeq C_{\mathrm{k},3}^{+}\times C_{\mathrm{k},3}^{-}\)
and \( C_{\mathrm{k},3}^{+}\simeq C_{\Gamma,3}\), hence \(\lvert C_{\mathrm{k},3}^{-}\rvert=3\).
  Since \(C_{\mathrm{k},3}\) is of type \(\left(9,3\right)\),
  we deduce that \(C_{\mathrm{k},3}^{-}\) is a cyclic group of order $3$
  and \(C_{\mathrm{k},3}^{+}\) is a cyclic group of order $9$.
  Therefore, we have
  \[u=1 \quad \text{and} \quad C_{\Gamma,3}\simeq\mathbb{Z}/9\mathbb{Z}.\]
Conversely, assume that \(u=1\) and \(C_{\Gamma,3}\simeq\mathbb{Z}/9\mathbb{Z}\).
By \cite[\S\ 14, Theorem 14.1, p. 232]{B-C}, we deduce that
\(|C_{\mathrm{k},3}|=\frac{1}{3}\cdot |C_{\Gamma,3}|^{2}\),
and so \( | C_{\mathrm{k},3}|=27\).
Furthermore, \( | C_{\mathrm{k},3}^{-}|=3\) and
\[C_{\mathrm{k},3}\simeq C_{\Gamma,3}\times C_{\mathrm{k},3}^{-}\simeq \mathbb{Z}/9\mathbb{Z}\times\mathbb{Z}/3\mathbb{Z}.\]
\end{proof}

\subsection{Proof of Theorem \ref{T39}}
Let $\Gamma=\mathbb{Q}(\sqrt[3]{d})$ be a pure cubic field,
where $d\ge 2$ is a cube-free integer,
$\mathrm{k}=\mathbb{Q}(\sqrt[3]{d},\zeta_3)$ be its normal closure, and
$C_{\mathrm{k},3}$ be the \(3\)-class group of \(\mathrm{k}\). \\
(1) Assume that the $3$- class group $C_{\mathrm{k},3}$ is of type $(9,3)$. We first write the integer $d$ in the form given by equation (3.2) of \cite[p. 55]{GERTH2}:
\begin{eqnarray}\label{eq2}
d &=& 3^e p_1^{e_1} \ldots p_v^{e_v}p_{v+1}^{e_{v+1}} \ldots p_w^{e_w}q_1^{f_1} \ldots q_I^{f_I}q_{I+1}^{f_{I+1}} \ldots q_J^{f_J},
\end{eqnarray}
where $p_i$ and $q_i$ are positive rational primes such that:

  \[  \left\{
  \begin{array}{l l}
    p_i  \equiv 1  \pmod  9,  & \quad  \text{ for} \quad 1\leq i\leq v, \\
    p_i  \equiv 4 \ \text{ or } \ 7  \pmod  9, & \quad \text{ for} \quad v+1 \leq  i \leq w,\\
    q_i  \equiv -1  \pmod  9, & \quad \text{ for} \quad 1\leq i\leq I,\\
    q_i  \equiv 2 \ \text{ or} \ 5  \pmod  9, & \quad \text{ for} \quad I+1\leq i\leq J,\\
    e_i =1 \ \text{ or} \ 2,  & \quad \text{ for} \quad 1\leq i\leq w,\\
    f_i =1 \ \text{ or} \ 2, & \quad \text{ for} \quad 1\leq i\leq J,\\
    e=0, 1 \ \text{ or} \  2.
  \end{array} \right.\]
  
 Let $C_{\mathrm{k},3}^{(\sigma)}$ be
the ambiguous ideal class group of $\mathrm{k}/\mathrm{k}_0 $, where $\sigma$ is a generator of $\operatorname{Gal}\left(\mathrm{k}/\mathrm{k}_0\right)$. 
It is known that $C_{\mathrm{k},3}^{(\sigma)}$
is an elementary abelian 3-group, because
an ambiguous class $\mathcal{C}$ satisfies $\sigma(\mathcal{C})=\mathcal{C}$, by definition,
and therefore
$\mathcal{C}^3 = \mathcal{C}\cdot\sigma(\mathcal{C})\cdot\sigma^{2}(\mathcal{C}) = \mathcal{N}_{k/k_0}(\mathcal{C}) = 1$,
since $\mathrm{k}_0$ has class number $1$.

The fact that the $3$-class group $C_{\mathrm{k},3}$ is of type $(9,3)$ implies that $\operatorname{rank}\,  C_{\mathrm{k},3}^{(\sigma)}=1$.
In fact, it is clear that if $C_{\mathrm{k},3}$ is of type $(9,3)$, then $\operatorname{rank}\,  C_{\mathrm{k},3}^{(\sigma)}=1 $ or $2$.

Let us assume that $\operatorname{rank}\,  C_{\mathrm{k},3}^{(\sigma)}=2$.
From \cite[\S\ 5, Theorem 5.3, pp. 97--98]{GERTH1}, we have
$$\operatorname{rank}\,C_{\mathrm{k},3}=2t-s,$$ 
where the integers $t$ and $s$ are defined in \cite[\S\ 5, Theorem 5.3, pp. 97--98]{GERTH1} as follows:
\begin{eqnarray*}
t &:= &\operatorname{rank}\,  C_{\mathrm{k},3}^{(\sigma)}, \\
s & : & \text {the rank of the matrix } (\beta_{i,j}) \text{ defined in \cite[\S\ 5, Theorem. 5.3, pp. 97--98]{GERTH1}}.
\end{eqnarray*}
Since $C_{\mathrm{k},3}$ is of type (9,3), then $\operatorname{rank}\,  C_{\mathrm{k},3}=2,$
and according to our hypothesis we have $t=\operatorname{rank}\,  C_{\mathrm{k},3}^{(\sigma)}=2$. So we get $s=2$. 

By \cite[\S\ 5, Theorem 5.3, pp. 97--98]{GERTH1},
the $3$-class group $C_{\mathrm{k},3}$ is isomorphic to the direct product of an abelian $3$-group of rank $2(t-s)$ and an elementary abelian $3$-group of rank $s$. 
Here $t=s=2$.
Thus $C_{\mathrm{k},3}$ is isomorphic to the direct product of an abelian $3$-group of rank $0$ and an elementary abelian $3$-group of rank $2$, we get 
 $$ C_{\mathrm{k},3} \simeq \mathbb{Z}/3\mathbb{Z }\times \mathbb{Z}/3\mathbb{Z}, $$
which contradicts the fact that $ C_{\mathrm{k},3}$ is of type $(9,3)$. 
We conclude $\operatorname{rank}\, C_{\mathrm{k},3}^{(\sigma)}=1$.

\paragraph{}

On the one hand,
suppose that $d$ is not divisible by any rational prime $p$ such that $p\equiv 1 \pmod  3$,
i.e. $w=0$ in Eq. (\ref{eq2}) above. According to \cite[\S\ 5, Theorem 5.1, p. 61]{GERTH2},
this implies that $C_{\mathrm{k},3}\simeq C_{\Gamma,3}\times C_{\Gamma,3}$.
Since the $3$-group $C_{\mathrm{k},3}$ is of type $(9,3)$, then $|C_{\mathrm{k},3}|=3^3$,
and by Lemma \ref{39} we have $C_{\Gamma,3} \simeq \mathbb{Z}/9\mathbb{Z}$,
so we get $|C_{\Gamma,3}\times C_{\Gamma,3}|=3^4$, which is a contradiction.
Hence, the integer $d$ is divisible by at least one rational prime $p$ such that $p\equiv 1 \pmod 3$.

On the other hand, the fact that the $3$-class group $C_{\mathrm{k},3}$ is of type $(9,3)$
implies that $C_{\Gamma,3} \simeq \mathbb{Z}/9\mathbb{Z}$ according to Lemma \ref{39}.
Since  $ h_{\Gamma} $ is exactly divisible by $9$,
then according to Lemma \ref{2p}, there are at most two primes congruent to $ 1 \pmod 3$ dividing $d$. 

Now, assume that there are exactly two different primes $p_1$ and $p_2$ dividing $d$ such that $p_1 \equiv p_2 \equiv 1 \pmod 3 $,
then, according to Lemma \ref{2prank}, we get:
   $$\Gamma^{\ast} = \Gamma_3^{(1)} = \Gamma\cdot M(p_1)\cdot M(p_2),$$
where $M(p_{1})$ (respectively, $M(p_{2})$) is the unique subfield of degree $3$ of $\mathbb{Q}(\zeta_{p_{1}})$ (respectively, $\mathbb{Q}(\zeta_{p_{2}})$).
If $M(p_{1})\neq M(p_{2})$, then $\Gamma\cdot M(p_1)$ and $\Gamma\cdot M(p_2)$ are two different subfields of $\Gamma_3^{(1)}$ over $\Gamma$ of degree $3$.
According to class field theory, we have $$\operatorname{Gal}(\Gamma_3^{(1)}/\Gamma) \cong C_{\Gamma,3} \simeq \mathbb{Z}/9\mathbb{Z}.$$ 
Since $\operatorname{Gal}(\Gamma_3^{(1)}/\Gamma)$ is a cyclic $3$-group,
there exist only one sub-group of $\operatorname{Gal}(\Gamma_3^{(1)}/\Gamma)$ of order $3$.
By the Galois correspondence, there exist a unique sub-field of $\Gamma_3^{(1)}$ over $\Gamma$ of degree $3$.
We conclude that $\Gamma\cdot M(p_{1}) = \Gamma\cdot M(p_{2})$, which is a contradiction.

Thus, $M(p_{1}) = M(p_{2})$, and then $p_1=p_2$, which contradicts the fact that $p_1$ and $p_2$ are two different primes.
Hence, there is exactly one prime congruent to $1 \pmod 3$ which divides $d$.
Thus the integer $d$ can be written in the following form:
 $$
d = 3^e p_1^{e_1}q_1^{f_1}\ldots q_I^{f_I}q_{I+1}^{f_{I+1}}\ldots q_J^{f_J},
$$
with $p_1  \equiv -q_i  \equiv 1  \pmod  3$, where $p_1$, $e$, $e_{1}$, $q_i$ and $f_{i}$ (for $ 1\leq i\leq J$) are defined in Eq. (\ref{eq2}).

Next, since $\operatorname{rank}\,  C_{\mathrm{k},3}^{(\sigma)}=1$,
then according to  \cite[\S\ 3, Lemma 3.1, p. 55]{GERTH2}, there are three possible cases as follows: \vspace{0.2cm} \\
  $\bullet$ \textbf{Case 1  :} \(2w+J=1\),\\
  $\bullet$ \textbf{Case 2  :} \(2w+J=2\),\\
  $\bullet$ \textbf{Case 3  :} \(2w+J=3\), 
  \vspace{0.2cm}\\
where $w$ and $J$ are the integers defined in Eq. (\ref{eq2}) above.
We will successively treat these cases as follows: \vspace{0.2cm}\\ 
$\bullet$ \textbf{Case 1:}
  We have $2w+J=1$, then \(w=0\) and \(J=1\).
  This case is impossible,
  because we have shown above that the integer \(d\) is divisible by exactly one prime number congruent to $1 \pmod 3$
  and thus $w = 1$.\\
 $\bullet$ \textbf{Case 2:}
  We have $2w+J=2$, and as in Case $1$, we necessarily have $w=1$ and $J=0$,
  which implies that \(d=3^{e}p_1^{e_{1}}\), where $p_1$ is a prime number such that \(p_1\equiv 1\pmod 3\), $e\in\{0,1,2\}$ and $e_1\in\{1,2\}$.
  Then,
   \begin{itemize}
   \item[$-$] If \(d \equiv \pm 1 \pmod  9\), then we necessarily have $e=0$.\\
   Assume that $p_1 \equiv 4 $ or $ 7 \pmod  9$, then $d  \not\equiv \pm 1 \pmod  9$ which is an absurd. So we necessarily have $p_1 \equiv 1 \pmod  9$. Thus the integer $d$ will be written in the form  $d=p_1^{e_{1}}$, where $p_1 \equiv 1 \pmod  9$ and $e_1\in\{1,2\}$.
   \item[$-$] If \(d \not \equiv \pm 1 \pmod  9\):
\begin{itemize}
\item[$*$] If $ e \neq 0$, the integer $d$ is written as $d=3^{e}p_1^{e_{1}}$, where $p_1 \equiv 1 \pmod  3$ and $e,e_1\in\{1,2\}$.
\item[$*$] If $ e = 0$, then $d$ is written as $d=p_1^{e_{1}}$ with $p_1 \equiv 4 $ or $ 7 \pmod  9$ and $e_1\in\{1,2\}$.
\end{itemize}   
   \end{itemize}
$\bullet$ \textbf{Case 3:}
We have $2w+J=3$, then we necessarily get \(w=1\) and \(J=1\), because $w \neq 0$. Thus
   \(d=3^{e}p_1^{e_{1}}q_1^{f_{1}}\), where \(p_1 \equiv 1 \pmod
    3 \), \(q_1 \equiv  -1 \pmod  3 \),  $e\in   \{0,1,2\}$ and $e_1, f_1 \in   \{1,2\}$. Then:
    \begin{itemize}
   \item[$-$] If \(d \equiv \pm 1 \pmod  9 \), we necessarily have $e=0$. If $p_1$ or $-q_1 \not \equiv  1 \pmod  9 $, then \(d \not \equiv \pm 1 \pmod  9 \) which is an absurd. It remain only the case where $p_1 \equiv -q_1 \equiv  1 \pmod  9 $. Then the integer $d$ will be written in the form $d=p_1^{e_{1}}q_1^{f_{1}}$, where $p_1 \equiv -q_1 \equiv  1 \pmod  9 $ and $e_1, f_1 \in   \{1,2\}$.

 \item[$-$] If \(d \not \equiv \pm1 \pmod  9 \): \\
  According to \cite[\S\ 5, p. 92]{GERTH1}, the rank of $C_{\mathrm{k},3}^{(\sigma)}$ is specified as follows:  $$ \operatorname{rank}\,   C_{\mathrm{k},3}^{(\sigma)}=t-2+q^{*},$$ where $t$ and $q^{*}$ are defined in the notations.\\
   On the one hand, the fact that \(\operatorname{rank}\,   C_{\mathrm{k},3}^{(\sigma)}=1\) imply that $t = 2$ or $3$ according to whether \(\zeta_{3}\) is norm of an element of $\mathrm{k}$ or not. \\ On the other hand, we have \(d=3^{e}p_1^{e_{1}}q_1^{f_{1}}\) with \(p_1 \equiv 1 \pmod
   3 \) and \(q_1 \equiv  -1 \pmod  3\).
By \cite[Chap. 9, Sec. 1, Proposition 9.1.4, p.110]{Clas}  we have $q_1$ is inert in $\mathrm{k}_0$, and by \cite[Sec 4, pp. 51-54]{DED} we have $q_1$ is ramifed in $\Gamma=\mathbf{Q}(\sqrt[3]{d})$. Since \(d \not \equiv \pm 1 \pmod 9 \), then  $3$ is ramifed in $\Gamma$  by \cite[Sec 4, pp. 51-54]{DED}, and  $3\mathcal{O}_{\mathrm{k}_0}=(\lambda)^2$ where $\lambda=1-\zeta_3.$ Since
$p_1 \equiv 1 \pmod  3$, then by \cite[Chap. 9, Sec. 1, Proposition 9.1.4, p.110]{Clas} we have $p_1=\pi_1\pi_2$ where \(\pi_1\) and \(\pi_2\)
are two primes of \(\mathrm{k}_{0}\) such that \(\pi_2=\pi_1^{\tau}\), the prime $p_1$ is ramifed in $\Gamma$, then $\pi_1$ and $ \pi_2$ are ramified in $\mathrm{k}$. Hence, the number of prime ideals which are ramified in $\mathrm{k}/\mathrm{k}_0$ is $t=4$, which contradicts the fact that $t=2$ or $3.$
  \end{itemize}

We summarize all forms of integer $d$ in the three cases $1, 2$ and $3$ as follows:
 \[ d = \left\{
  \begin{array}{l l l}
   p_1^{e_1}    & \quad \text{with \ $ p_1 \equiv 1  \pmod  9, $  }\\
   p_1^{e_1}    & \quad \text{with \ $ p_1 \equiv 4 \ \text{or} \ 7  \pmod  9, $  }\\
   3^{e}p_1^{e_1} \not \equiv \pm 1  \pmod 9 & \quad \text{with \ $ p_1 \equiv 1 \pmod  9,$  }\\
   3^{e}p_1^{e_1} \not \equiv \pm 1  \pmod 9 & \quad \text{with \ $ p_1 \equiv 4 \ \text{or} \ 7 \pmod  9,$  }\\
   p_1^{e_1}q_1^{f_1} \equiv \pm 1 \pmod 9 & \quad \text{with \ $ p_1 \equiv -q_1 \equiv 1  \pmod  9, $     }\\

  \end{array} \right.\]
where $e,e_1, f_1 \in   \{1,2\}$.\\

 Our next goal is to show that the only possible form of the integer $d$ is the first form $d=p_1^{e_1}$, where $p_1 \equiv 1 \pmod 9$ and $e_1=1$ or $2$. \vspace{0.2cm} \\
 $\bullet$ Case where $d=p_1^{e_1}$, with $ p_1 \equiv 4 \ \text{or} \ 7 \pmod  9$:  
 \begin{itemize}
 \item[$-$] If $\left(\frac{3}{p}\right)_3\neq 1$, then  according to \cite[\S\ 1, Conjecture 1.1, p. 1]{CP}, we have $C_{\mathrm{k},3}\simeq \mathbb{Z}/3\mathbb{Z}$, which contradict the fact that $C_{\mathrm{k},3}$ is of type $(9,3)$. We note that in this case, the fields $\Gamma$ and $\mathrm{k}$  are of Type III, respectively, $\alpha$, in the terminology of \cite[ \S\ 15, Theorem 15.6, pp. 235-236]{B-C},  respectively, \cite[\S\ 2.1, Theorem 2.1, p. 4]{AMITA}. \\
 \item[$-$] If $\left(\frac{3}{p}\right)_3 = 1$, then  according to \cite[\S\ 1, Conjecture 1.1, p. 1]{CP}, we have $C_{\mathrm{k},3}\simeq \mathbb{Z}/3\mathbb{Z}\times \mathbb{Z}/3\mathbb{Z}$, which is impossible. We note that in this case, the fields $\Gamma$ and $\mathrm{k}$ are of Type I,  respectively, $\beta$, in the terminology of \cite[ \S\ 15, Theorem 15.6, pp. 235-236]{B-C},  respectively, \cite[\S\ 2.1, Theorem 2.1, p. 4]{AMITA}. \\
 For this case, we see that in \cite[ \S\ 17, Numerical Data, p. 238]{B-C}, and also in the tables of \cite{TABdata} which give  the class number of a pure cubic field, the prime numbers $p=61, 67, 103,$ and $151$, which are all congruous to $4 $ or $ 7 \ \pmod 9$, verify the following  properties:
\begin{itemize}
\item[(i)] $3$ is a residue cubic modulo $p$;
\item[(ii)] $3$ divide exactly the class number of  $\Gamma$;
\item[(iii)]  $u=3$;
\item[(iv)] $C_{\Gamma,3} \simeq \mathbb{Z}/3 \mathbb{Z}$, and $C_{\mathrm{k},3}\simeq \mathbb{Z}/3 \mathbb{Z}\times \mathbb{Z}/3 \mathbb{Z}.$
\end{itemize}
 \end{itemize}
$\bullet$ Case where $d=3^{e}p_1^{e_1} \not \equiv \pm 1 \pmod 9$, with $ p_1 \equiv 1 \pmod  9$:
$\\ $
Here $e,e_1 \in   \{1,2\}$. As $\mathbb{Q}(\sqrt[3]{ab^2})=\mathbb{Q}(\sqrt[3]{a^2b})$,
we can choose $e_1=1$, i.e. $d=3^{e}p_1$ with $e\in \lbrace 1,2 \rbrace$. On the one hand, the fact that  $ p_1 \equiv 1 \pmod  3$ implies by  \cite[Chap. 9, Sec. 1, Proposition 9.1.4, p.110]{Clas} that  $p_1=\pi_1 \pi_2$ with $\pi_1^{\tau}=\pi_2$ and $\pi_1 \equiv \pi_2 \equiv 1 \pmod {3\mathcal{O}_{\mathrm{k}_0}}$, the prime $p_1$ is totally ramified in $\Gamma$, then $\pi_1$ and $\pi_2$ are totally ramified in $\mathrm{k}$ and we have $\pi_1\mathcal{O}_{\mathrm{k}}=\mathcal{P}_1^3$
and $\pi_2 \mathcal{O}_{\mathrm{k}}=\mathcal{P}_2^3$, where $\mathcal{P}_1, \mathcal{P}_2$ are two prime ideals of $\mathrm{k}$.

We know that $3\mathcal{O}_{\mathrm{k}_0}=\lambda^2\mathcal{O}_{\mathrm{k}_0}$, with $\lambda=1-\zeta_3$ a prime element of $\mathrm{k}_0$. Since $d \not \equiv \pm 1 \pmod 9 $, then $3$ is totally ramified in $\Gamma$, and then $\lambda$ is ramified in $\mathrm{k}/\mathrm{k}_0$. Hence, the number of ideals which are ramified in $\mathrm{k}/\mathrm{k}_0$ is $t=3.$

On the other hand, from \cite[Chap. 9, Sec. 1, Proposition 9.1.4, p.110]{Clas} we have $3=-\zeta_3^2\lambda^2$, then $\mathrm{k}=\mathrm{k}_0(\sqrt[3]{x})$ with $x=\zeta_3^2\lambda^2 \pi_1 \pi_2$. The primes $\pi_1$ and $\pi_2$ are congruent to $ 1 \pmod {\lambda^3}$ because $ p_1 \equiv 1 \pmod  9$, then according to \cite[\S\ 3, Lemma 3.3, p. 17]{AMITA} we have $\zeta_3$ is a norm of an element of $\mathrm{k} \setminus \lbrace 0\rbrace$, so $q^{*}=1$. We conclude according to \cite[\S\ 5, p. 92]{GERTH1} that $\operatorname{rank}\,  C_{\mathrm{k},3}^{(\sigma)}=2$ which is impossible. \\
$\bullet$ Case where $d=3^{e}p_1^{e_1} \not \equiv \pm 1 \pmod 9$, with $ p_1 \equiv 4 \ \text{ or } \ 7  \pmod  9$: \\
Here $e,e_1 \in   \{1,2\}$.
So we can choose $e=1$, then $d=3p_1^{e_1}$ with $e_1\in \lbrace 1,2 \rbrace$.
As above we get $\pi_1$ and $\pi_2$, and $\lambda$ are ramified in $\mathrm{k}/\mathrm{k}_0$. \\
 Put $p\mathcal{O}_{\Gamma}=\mathcal{P}^3$,  $\pi_1\mathcal{O}_{\mathrm{k}}=\mathcal{P}_1^3$, $\pi_2 \mathcal{O}_{\mathrm{k}}=\mathcal{P}_2^3$ and $\lambda \mathcal{O}_{\mathrm{k}}=I^3$, where $\mathcal{P}$ is a prime ideal of $\Gamma$, and $\mathcal{P}_1, \mathcal{P}_2$  and $I$ are prime ideals of $\mathrm{k}$. According to \cite[\S\ 3.2, Theorem 3.5, pp 36-39]{Is}
we get $C_{k,3}$ is cyclic of order $3$ which contradict the fact that $C_{k,3}$ is of type $(9,3)$.\\
$\bullet$  Case where  $d=p_1^{e_1}q_1^{f_1} \equiv \pm1 \pmod 9$, with $ p_1 \equiv -q_1 \equiv 1  \pmod  9$ :
$\\ $
Since $d \equiv \pm1 \pmod  9$, then according to \cite[Sec. 4, pp. 51-54]{DED} we have $3$ is not ramified in the field $\Gamma$, so $\lambda=1-\zeta_3$ is not ramified in $\mathrm{k}/\mathrm{k}_0$.
 As  $ p_1 \equiv 1 \pmod  3$, then by \cite[Chap. 9, Sec. 1, Proposition 9.1.4, p.110]{Clas} $p_1=\pi_1 \pi_2$ with $\pi_1^{\tau}=\pi_2$ and $\pi_1 \equiv \pi_2 \equiv 1 \pmod {3\mathcal{O}_{\mathrm{k}_0}}$, so $\pi_1$ and $\pi_2$ are totally ramified in $\mathrm{k}$. Since  $ q_1 \equiv -1 \pmod  3$, then $q_1$ is inert in $\mathrm{k}_0$. Hence, the primes ramifies in $\mathrm{k}/\mathrm{k}_0$ are $\pi_1, \pi_2$ and $q_1$.  Put $x=\pi_1^{e_1} \pi_2^{e_1} \pi^{f_1}$, where $-q_1=\pi$ is a prime number of $\mathrm{k}_0$, then we have
 $\mathrm{k}=\mathrm{k}_0(\sqrt[3]{x})$.  The fact that $ p_1 \equiv -q_1 \equiv 1  \pmod  9$ imply that $\pi_1 \equiv \pi_2 \equiv \pi \equiv 1 \pmod {\lambda^3}$, then by  \cite[\S\ 3, Lemma 3.3, p. 17]{AMITA}, $\zeta_3$ is a norm of an element of $\mathrm{k}\setminus \lbrace 0\rbrace$, so $q^{*}=1$. Then by \cite[\S\ 5, p. 92]{GERTH1} we get $\operatorname{rank}\,   C_{\mathrm{k},3}^{(\sigma)}=t-2+q^{*}=2$ which is impossible.\\  

Finally, we have shown that if the $3$-class group $C_{\mathrm{k},3}$ is of type $(9,3)$, then $d=p^{e}$, where $p$ is a prime number such that $p \equiv 1 \pmod 9$ and $e=1$ or $2$. We can see that this result is compatible with the first form of the integer $d$ in \cite[\S\ 1, Theorem 1.1, p. 2]{AMITA}.\\
 (2)  Suppose that $d=p^{e}$, with $p$ is a prime number congruent to $1 \pmod 9$. Here $e \in \lbrace 1, 2 \rbrace $, since $\mathbb{Q}(\sqrt[3]{p})=\mathbb{Q}(\sqrt[3]{p^2})$ we can choose $e=1$.
 From \cite[\S\ 14, Theorem $14.1$, p. 232]{B-C}, we have
$h_{\mathrm{k}}=\frac{u}{3}\cdot h_{\Gamma}^{2}$, the fact that $u=1$ and that $ h_{\Gamma} $ is exactly divisible by $9$ implies that  $h_{\mathrm{k}}$  is exactly divisible by $27$. \\
Since $ h_{\Gamma} $ is exactly divisible by $9$, then by \cite[Lemma $5.11$]{Cal-Emer} we have $\operatorname{rank}\, C_{\mathrm{k},3}=2$. We conclude that $C_{\mathrm{k},3} \simeq \mathbb{Z}/9\mathbb{Z}\times
\mathbb{Z}/3\mathbb{Z}$.


\section{Numerical Examples}\label{Appendix}
Let $\Gamma=\mathbb{Q}(\sqrt[3]{ab^2})$ be a pure cubic field,
where $a$ and $b$ are coprime square-free integers.
We point out that $\mathbb{Q}(\sqrt[3]{ab^2})=\mathbb{Q}(\sqrt[3]{a^2b})$.
Assume that $\mathrm{k}=\mathbb{Q}(\sqrt[3]{ab^2},\zeta_3)$ and $C_{\mathrm{k},3}$ is of type $(9,3)$.
Using the system Pari \cite{PARI}, we compute class groups
for $b=1$ and a prime $a=p\equiv 1\pmod 9$.
The following table illustrates our main result Theorem \ref{T39} and Lemma \ref{39}.
Here we denote by:
\begin{eqnarray*}
C_{\Gamma,3} \ (\text{respectively,} \ C_{\mathrm{k},3}) &:&  \text{the 3-class group of} \ \Gamma \ (\text{respectively,} \ \mathrm{k});\\
h_{\Gamma,3} \ (\text{respectively,} \ h_{\mathrm{k},3}) &:&  \text{the 3-class number of} \ \Gamma \ (\text{respectively,} \ \mathrm{k}).
\end{eqnarray*}

\begin{center}
Table : Some fields $\mathbb{Q}(\sqrt[3]{p},\zeta_3)$ whose $3$-class group is of type \((9,3)\).
\end{center}
\begin{longtable}{| c | c | c | c | c | c | c | c |   }
\hline
  $p$ &  $p^2$ & $p \pmod  9$  & $h_{\Gamma,3}$  & $h_{\mathrm{k},3}$ & $u$ & $C_{\Gamma,3}$ & $C_{\mathrm{k},3}$ \\
\hline
\endhead
$199$   & $39601$  & $1$  & $9$ & $27$ & $1$ & $[9]$ & $[9, 3]$ \\
  $271$   & $73441$  & $1$  & $9$ & $27$ & $1$ & $[9]$ & $[9, 3]$ \\
  $487$   & $237169$  & $1$  & $9$ & $27$ & $1$ & $[9]$ & $[9, 3]$ \\
  $523$   & $273529$  & $1$  & $9$ & $27$ & $1$ & $[9]$ & $[9, 3]$ \\
  $1297$   & $1682209$  & $1$  & $9$ & $27$ & $1$ & $[9]$ & $[9, 3]$ \\
  $1621$   & $2627641$  & $1$  & $9$ & $27$ & $1$ & $[9]$ & $[9, 3]$ \\
  $1693$   & $2866249$  & $1$  & $9$ & $27$ & $1$ & $[9]$ & $[9, 3]$ \\
  $1747$   & $3052009$  & $1$  & $9$ & $27$ & $1$ & $[9]$ & $[9, 3]$ \\
  $1999$   & $3996001$  & $1$  & $9$ & $27$ & $1$ & $[9]$ & $[9, 3]$ \\
  $2017$   & $4068289$  & $1$  & $9$ & $27$ & $1$ & $[9]$ & $[9, 3]$ \\
  $2143$   & $4592449$  & $1$  & $9$ & $27$ & $1$ & $[9]$ & $[9, 3]$ \\
  $2377$   & $5650129$  & $1$  & $9$ & $27$ & $1$ & $[9]$ & $[9, 3]$ \\
  $2467$   & $6086089$  & $1$  & $9$ & $27$ & $1$ & $[9]$ & $[9, 3]$ \\
  $2593$   & $6723649$  & $1$  & $9$ & $27$ & $1$ & $[9]$ & $[9, 3]$ \\
  $2917$   & $8508889$  & $1$  & $9$ & $27$ & $1$ & $[9]$ & $[9, 3]$ \\
  $3511$   & $12327121$  & $1$  & $9$ & $27$ & $1$ & $[9]$ & $[9, 3]$ \\
  $3673$   & $13490929$  & $1$  & $9$ & $27$ & $1$ & $[9]$ & $[9, 3]$ \\
  $3727$   & $13890529$  & $1$  & $9$ & $27$ & $1$ & $[9]$ & $[9, 3]$ \\
  $3907$   & $15264649$  & $1$  & $9$ & $27$ & $1$ & $[9]$ & $[9, 3]$ \\
  $4159$   & $17297281$  & $1$  & $9$ & $27$ & $1$ & $[9]$ & $[9, 3]$ \\
  $4519$   & $20421361$  & $1$  & $9$ & $27$ & $1$ & $[9]$ & $[9, 3]$ \\
  $4591$   & $21077281$  & $1$  & $9$ & $27$ & $1$ & $[9]$ & $[9, 3]$ \\
  $4789$   & $22934521$  & $1$  & $9$ & $27$ & $1$ & $[9]$ & $[9, 3]$ \\
  $4933$   & $24334489$  & $1$  & $9$ & $27$ & $1$ & $[9]$ & $[9, 3]$ \\
  $4951$   & $24512401$  & $1$  & $9$ & $27$ & $1$ & $[9]$ & $[9, 3]$ \\
  $5059$   & $25593481$  & $1$  & $9$ & $27$ & $1$ & $[9]$ & $[9, 3]$ \\
  $5077$   & $25775929$  & $1$  & $9$ & $27$ & $1$ & $[9]$ & $[9, 3]$ \\
  $5347$   & $28590409$  & $1$  & $9$ & $27$ & $1$ & $[9]$ & $[9, 3]$ \\
  \hline
\end{longtable}

\begin{remark} Let $p$ is a prime number such that $p \equiv 1 \pmod 9$.
Let $\Gamma=\mathbb{Q}(\sqrt[3]{p})$,  $\mathrm{k}=\mathbb{Q}(\sqrt[3]{p},\zeta_3)$ be the normal closure of the pure cubic field $\Gamma$
and $C_{\mathrm{k},3}$ be the $3$-part of the class group of $\mathrm{k}$. If $9$ divide exactly the class number of $\mathbb{Q}(\sqrt[3]{p})$ and $u=1$, then according to  Theorem 1.1, the $3$-class group of $\mathbb{Q}(\sqrt[3]{p},\zeta_3)$ is of type $(9,3).$ Furthermore, if $3$ is not residue cubic modulo $p$, then a generators of $3$-class group of $\mathbb{Q}(\sqrt[3]{p},\zeta_3)$ can be deduced by \cite[\S\ 3, Theorem 3.2, p. 10]{GENArXiv}.
\end{remark}

\section{Acknowledgements} 
The authors would like to thank Professors Daniel C. Mayer and Mohammed Talbi who were of a great help concerning correcting the spelling mistakes
that gave more value to the work.

\begin{quote}

Abdelmalek AZIZI, Moulay Chrif ISMAILI and Siham AOUISSI \\
Department of Mathematics and Computer Sciences, \\
Mohammed first University, \\
60000 Oujda - Morocco, \\
abdelmalekazizi@yahoo.fr, mcismaili@yahoo.fr,  aouissi.siham@gmail.com. \\

Mohamed TALBI \\
Regional Center of Professions of Education and Training, \\
60000 Oujda - Morocco, \\
ksirat1971@gmail.com. \\

\end{quote}

\end{document}